\documentclass[11pt, reqno]{amsart}
\usepackage{amscd}        % Package used to produce simple commutative diagrams
\usepackage{hyperref}
\usepackage{graphicx}
\usepackage{rotating}
\usepackage{amssymb}
\usepackage{epstopdf}
\usepackage{tikz}
\usepackage{dirtytalk}
\usepackage[inner=1.3in,outer=1.3in,bottom=1.5in, top=1.5in]{geometry}

\input xypic
\xyoption{all}
\input epsf.tex

\usepackage{epsfig}
\usepackage{enumerate}
\usepackage{graphicx}% http://ctan.org/pkg/graphicx
\usepackage{yhmath}% http://ctan.org/pkg/yhmath
\usepackage{mathdots}% http://ctan.org/pkg/mathdots
\usepackage{mathtools} 
\usepackage{pifont}

\usepackage{tikz}
\usetikzlibrary{matrix,arrows,decorations.pathmorphing}
\usepackage{tikz-cd}

\numberwithin{equation}{section}

\newtheorem{theorem}{Theorem}[section]
\newtheorem{lemma}[theorem]{Lemma}
\newtheorem{proposition}[theorem]{Proposition}
\newtheorem{corollary}[theorem]{Corollary}
\newtheorem{conjecture}[theorem]{Conjecture}

\theoremstyle{definition}
\newtheorem{remark}[theorem]{Remark}
\newtheorem{definition}[theorem]{Definition}

\newtheorem{question}{Question}
\newtheorem*{teo}{Theorem}

\def \P { \mathbb{P}}
\def\C{{\mathbb C}}

\def \F { \mathbb{F}}
\def\Z{{\mathbb Z}}
\def\Q{{\mathbb Q}}

\def \gal { {\rm Gal}}

\begin{document}

\title{Density questions on arithmetic equivalence}
\author{}
%\date{}                                           % Activate to display a given date or no date

\author{Guillermo Mantilla-Soler}\thanks{G. Mantilla Soler's work was supported in part by the Aalto Science Institute}

\begin{abstract}
It is a classic result that two number fields have equal Dedekind zeta functions if and only if  the arithmetic type of a prime $p$ is the same in both fields for almost all prime $p$. Here, almost all means with the possible exception of a set of Dirichlet  density zero. One of the results of this paper shows that the condition density zero can be improved to a specific positive density that depends solely in the degree of the fields. More specifically, for every positive $n$ we exhibit a positive constan $c_{n}$ such that any two degree $n$ number fields $K$ and $L$ are arithmetically equivalent if and only if the set of primes $p$ such that the arithmetic type of $p$ in $K$ and $L$ is not the same has Dirichlet  density at most $c_n$. We in fact show that $\displaystyle c_n=\frac{1}{4n^2}$ works and give a heuristic evidence that points to the fact that this value might be improved to $\displaystyle \frac{2}{n^2}$.  We also show that to check whether or not two number fields are arithmetically equivalent it is enough to check equality between finitely many coefficients of their zeta functions, and we give an upper bound for such number.
\end{abstract}

\maketitle

\section{Introduction and statement of results}

Two number fields are called {\it Arithmetically equivalent} if they have the same Dedekind zeta function. Arithmetically equivalence is as strong relation between number fields, for instance any two A.E fields have the same discriminant, unit group, signature, Galois closure, the product of class number times regulator and others (see \cite{Perlis1} or \cite[\S 2]{Manti}.) A  well known characterization for arithmetic equivalence can be given in terms of residue class degrees of rational primes. 

\begin{definition}
The {\it arithmetic type} of a rational prime $\ell$ in $K$ is the ordered tuple $A_{\ell}(K):=(f_{1}, \ldots, f_{g})$ where $f_{1} \leq \cdots \leq f_{g}$ are the residue class degrees of $\ell$ in $K$.
\end{definition}

\begin{theorem}[Perlis \cite{Perlis1}]\label{Perlis}
Let $K$ and $L$ be number fields. The following are equivalent:

\begin{itemize}
\item[(a)] The fields $K$ and $L$ have the same zeta functions:  \[\zeta_{K}(s)=\zeta_{L}(s).\]

\item[(b)] For all rational prime $p$ the arithmetic types of $p$ in $K$ and $L$ coincide: \[A_{p}(K)=A_{p}(L).\]

\item[(c)] For almost every rational prime $p$ the arithmetic types of $p$ in K and L coincide: \[A_{p}(K)=A_{p}(L).\]

\end{itemize}
\end{theorem}

\noindent The  \say{almost all} means that the set of primes  \[\{p :  A_{p}(K) = A_{p}(L)\}\] has Dirichelt density equal to $1$.  It is natural to ask if such density condition can be improved. Let $K,L$ be number fields and let \[A_{K,L}:=\{ p : A_{p}(K)=A_{p}(L) \}\setminus  \text {ramified primes in $KL$}.\]  The statement of Perlis' theorem written in terms of the set $A_{K,L}$ reads as \[\delta(A_{K,L}) <1 \iff \zeta_{K}(s) \neq \zeta_{L}(s).\]

\begin{question}\label{HayCota}

 Can the bound of $1$ be improved? In other words, \[ \mbox{is there} \ c_{K,L} < 1   \ \mbox{such that} \   \delta(A_{K,L})  \leq  c_{K,L}  \rightarrow \zeta_{K}(s) \neq \zeta_{L}(s).\]

\end{question}

One of the main results of this paper is that indeed such a bound exists, moreover it is explicit and depends only on the degree of the fields.

\begin{teo}[Theorem \ref{ElPrincipal}]
Let $n >1$ be an integer and let $K, L$ be two degree $n$ number fields. Then, $K$ and $L$ are arithmetically equivalent if and only if the set of primes such that their arithmetic type is not the same in $K$ and $L$ has Dirichelt Density less that $\displaystyle \frac{1}{4n^2} $, i.e.,   \[  \delta(A_{K,L}) >1-\frac{1}{4n^2} \iff \zeta_{K}(s)=\zeta_{L}(s).\]

\end{teo}

We in fact give some heuristics evidence, and prove some particular cases, showing that the constant  $\displaystyle \frac{1}{4n^2}$ could be improved to  $\displaystyle \frac{2}{n^2}.$

Another interesting related question about Arithmetically equivalence is the following: Suppose we have a number field $K$, and say we want to know if a given number field $L$ is arithmetically equivalent to it. 
Could we test the arithmetic type for a set of finitely many primes and conclude from such finite sets whether or not $K$ and $L$ are A.E? Thanks to Proposition \ref{LasClausurasDelChamoGalois} we now that a necessary condition for $L$ to be arithmetically equivalent to $K$ is that they have the same Galois closure, hence we may assume that this is the case.  More explicitly, we ask

\begin{question}\label{CheboEfectivo} Let $K$ be a number field. Is there a positive constant $X_{K}$ such that for every number field $L$ with the same Galois closure as $K$  \[ \zeta_{K}(s)=\zeta_{L}(s)  \iff A_{p}(K)=A_{p}(L)\  \mbox{for all}  \ p \leq X_{K}? \]

\end{question}

We answer this question in the positive and provide a constant for it. 

\begin{teo}[Theorem \ref{AECheboEfectiva}] Let $K$ be a number field.Then, for every number field $L$ such that $\widetilde{K}=\widetilde{L}$ we have that \[ \zeta_{K}(s)=\zeta_{L}(s)  \iff A_{p}(K)=A_{p}(L)\  \mbox{for all}  \ p \leq |{\rm Disc}(\widetilde{K})|^{12577}.\] Furthermore, under GRH the above bound can be replaced by \[(4\log(|{\rm Disc}(\widetilde{K})|)+2.5[\widetilde{K}:\Q]+5)^2\]
 \end{teo}

\section{Preliminary definitions and results}

Let $E$ be a number field and let $n$ be a positive integer.  Let $a_{n}(E)$ be the $n$-th coefficient of the Dedekind zeta function of $E$, i.e., $\displaystyle a_{n}(E):=\#\{ I \unlhd O_{E} :  \| I \|=n \}$. It follows from basic principles of Dirichlet's series that two number fields $K$ and $L$ are arithmetically equivalent if and only if $\displaystyle a_{n}(K)= a_{n}(L)$ for all non-negative integer $n$.\\

One of the first invariants of a number field that is captured by its Dedekind zeta function is the degree. Usually, see for instance  \cite[Theorem 1]{Perlis1}, the invariance of the degree is proved using the group theoretical characterization of arithmetic equivalence. Here we give a slightly different argument based on the existence of splitting primes. This we do since it motivates our focus of the set of splitting primes and the generalizations we will define later in the paper.

\begin{lemma}\label{AEImpliesSameDeg}

Let $K,L$ be number fields  such that $\zeta_{K}(s)=\zeta_{L}(s)$. Then, $[K:\Q]=[L:\Q]$.

\end{lemma}

\begin{proof}
Let $p$ be a prime that splits in $K$. Then, $[K:\Q]=a_{p}(K)$. Since $\zeta_{K}(s)=\zeta_{L}(s)$, $[K:\Q]=a_{p}(K)=a_{p}(L) \le [L:\Q]$. By symmetry of the argument we have that $[L:\Q] \le [K:\Q]$.
\end{proof}

Given a number field $E$ we denote by $S_{E}$ the set of rational primes that split in $E$. An important feature of the set of splitting primes is that it determines the Galois closure of a field.  In particular, since \begin{equation}\label{EcuaSplitDeg}  S_{E}=\{p: a_{p}(E)=[E:\Q]\}\end{equation} we have that the Dedekind zeta function determines the Galois closure. We denote by $\widetilde{E}$ the Galois closure of a number field.

\begin{proposition}\label{LasClausurasDelChamoGalois}
Let $K,L$ be number fields  such that $\zeta_{K}(s)=\zeta_{L}(s)$. Then,  $\widetilde{K}= \widetilde{L}$  
\end{proposition}

\begin{proof}
Thanks to Lemma \ref{AEImpliesSameDeg}  we know that $n:=[K:\Q]=[L:\Q]$. Hence,  $S_{\widetilde{K}}=S_{K}=\{p: a_{p}(K)=n\}=\{p: a_{p}(L)=n\}=S_{L}=S_{\widetilde{L}}$. Therefore, $S_{\widetilde{K}}=S_{\widetilde{L}}$ which implies that $\widetilde{K}= \widetilde{L}$  since they are Galois number fields.
\end{proof}

\begin{definition}
Let $K, L$ be number fields let ${\rm Ram}_{K,L}$ be the set of rational primes that ramify in either $K$ or $L$.  
\begin{eqnarray*}
% A_{K,L}:=& \{ p : A_{p}(K)=A_{p}(L) \}\setminus {\rm Ram}_{K,L} \\
S_{K,L}:=& \ \ \ \ \{ p : \#A_{p}(K)=\#A_{p}(L) \}\setminus {\rm Ram}_{K,L} \\
T_{K,L}:= & \{ p : a_{p}(K)=a_{p}(L) \}\setminus {\rm Ram}_{K,L}
\end{eqnarray*}
\end{definition}

Let $K$, $L$ be number fields of the same degree. It follows from equation \ref{EcuaSplitDeg} that $S_{K} \cap S_{L} \subseteq T_{K,L}$, equivalently \begin{equation} S_{\widetilde{K}} \cap S_{\widetilde{L}} \subseteq T_{K,L}. \end{equation}

This gives a positive lower bound of $\delta(A_{K,L})$. More precisely, 

\begin{proposition}\label{PrimerDesigualdadA}
Let $K,L$ be number fields of the same degree. Then, \[\frac{1}{[\widetilde{K}\widetilde{L}:\Q]} \leq  \delta(T_{K,L}).\]
\end{proposition}

\begin{proof}
By the compatibility of Frobenius elements in Galois extensions, see Lemma below, $S_{\widetilde{K}\widetilde{L}} =S_{\widetilde{K}}\cap S_{\widetilde{L}}$. Thus, $S_{\widetilde{K}\widetilde{L}} \subseteq T_{K,L}$ and thanks to the Chebotarev's density theorem the result follows. 
\end{proof}

\begin{lemma}\label{SplittingPrimesGaloisClosure}
Let $K,L$ be number fields. Then, \[S_{\widetilde{K}\widetilde{L}} =S_{\widetilde{K}}\cap S_{\widetilde{L}}.\]
\end{lemma}

\begin{proof}
Since $S_{E}=S_{\widetilde{E}}$ for any number field $E$, and since the compositum of the Galois closures is the Galois closure of the compositum, we may assume that $K$ and $L$ are Galois number fields. In that situation the statement is equivalent to say that a rational prime $p$ splits in $KL$ if and only if it splits in $K$ and in $L$. This follows from the fact that Frobenius elements are compatible under restriction maps and that the product restriction map
\[\gal(KL/\Q) \to \gal(K/\Q) \times \gal(L/\Q)\] is injective. \end{proof}

Next we define the sets $S_{E}(m)$ in order to generalize the notion of splitting set of primes on a number field $E$

\begin{definition}
Let $E$ be a degree $n$ number field with discriminant $\Delta_{E}$, and let $m$ be a non-negative integer not bigger than $n$.  We define the set \[S_{E}(m):=\{p : a_{p}(E)=m; p \nmid \Delta_{E}\}.\]
\end{definition} 

\begin{remark}
Notice that $S_{E}(n)=S_{E}$, thus this definition is meant to generalize the splitting set $S_{E}$.
\end{remark}

\begin{proposition}\label{SplitEnGalois}
Let $E$ be a degree $n$ number field. Suppose that $E/\Q$ is Galois. If $m$ is a positive integer strictly less than $n$,   $S_{E}(m)=\emptyset$.  In other words, up to ramified primes, the set of prime numbers $\P$ is partitioned as $S_{E}(0) \sqcup S_{E}$. In particular,  up to ramification, \[S_{E}(0)=\{p: a_{p}(E) < n\}\] and its density is equal to  \[\delta(S_{E}(0))=1-\delta (S_{E}) =\frac{n-1}{n}.\]

\end{proposition}

\begin{proof}
Let $p$ be a rational prime that is unramified in $E$ and such that $a_{p}(E) \neq 0$. By definition there is an ideal $\mathfrak{P}$ in $O_{E}$ with norm $p$. In particular, $\mathfrak{P}$ is maximal and  the inertia degree of $\mathfrak{P}$ over $p$ is $1$. Given that the $E/\Q$ is Galois and $p$ is unramified we conclude that the decomposition group of $\mathfrak{P}$ over $p$ is trivial. Hence $p$ splits, so there are $n$ maximal ideals lying over $p$. This shows that if $a_{p}(E) \neq 0$ then $a_{p}(E)=n$ proving the first claim. The second claim follows from the first, the additivity of Dirichlet's density and the Chebotarev's density theorem.
\end{proof}

\begin{corollary}\label{MotivaElDelta}

Let $K, L$ be degree $n$ Galois number fields. Then, \[T_{K,L}=\left( S_{K}  \cap S_{L} \right) \bigcup \left(\P \setminus  S_{K} \cap  \P \setminus  S_{L} \right).\] Furthermore, \[\delta(T_{K,L})=\frac{2}{[KL:\Q]} +1 -\frac{1}{K:\Q]} -\frac{1}{[L:\Q]}.\]
\end{corollary}

\begin{proof}
Thanks to Proposition \ref{SplitEnGalois}  if a prime $p$ belongs to $T_{K,L}$ then either $a_{p}(K)=a_{p}(L)=n$ or $a_{p}(K)=a_{p}(L)<n$ from which the first claim follows. Since this is a disjoint union we have that \[\delta(T_{K,L})=\delta(S_{K}  \cap S_{L} ) + \delta\left(\P \setminus  S_{K} \cap  \P \setminus  S_{L} \right)= \delta(S_{K}  \cap S_{L} )+1-\delta(S_{K}  \cup S_{L} ) =\]
\[2\delta(S_{K}  \cap S_{L} )+1-\delta(S_{K})-\delta(S_{L})=\frac{2}{[KL:\Q]} +1 -\frac{1}{K:\Q]} -\frac{1}{[L:\Q]}.\]
The last equality follows from Lemma \ref{SplittingPrimesGaloisClosure} and the Chebotarev's density theorem.
\end{proof}

For general degree $n$ number fields $K$ and $L$ an explicit description for $\delta(T_{K,L})$ follows from the fact that \[T_{K,L}=\bigcup_{m=0}^{n}S_{K}(m)\cap S_{L}(m)\] so that 
\[\delta(T_{K,L})= \sum_{m=0}^{n} \delta(S_{K}(m)\cap S_{L}(m)).\] It seems  hard however, for an arbitrary $m$ and no hypothesis on $K$ and $L$, to calculate the single term $\delta(S_{K}(m)\cap S_{L}(m)$. At this point the best we can do is to give an upper bound for $\delta(T_{K,L})$ following the ideas of Corollary \ref{MotivaElDelta}

\begin{proposition}\label{ALessThanD}
Let $K$ and $L$ be degree $n$ number fields. Then, \[\delta(T_{K,L}) \leq \frac{2}{[\widetilde{K}\widetilde{L}:\Q]} +1 -\frac{1}{[\widetilde{K}:\Q]} -\frac{1}{[\widetilde{L}:\Q]}.\] If both $K$ and $L$ are Galois, the inequality becomes an equality.
\end{proposition}

\begin{proof}
Let $m$ be a non-negative integer strictly less than $n$. Then, $S_{E}(m) \subseteq \P \setminus S_{E}$ for $E$ equal to either $K$ or $L$. Therefore for such values of $m$, \[S_{K}(m) \cap S_{L}(m)  \subseteq \P \setminus S_{K} \cap \P \setminus S_{L}.\] In particular, \[T_{K,L}=\bigcup_{m=0}^{n}S_{K}(m)\cap S_{L}(m) \subseteq \left(\P \setminus S_{K} \cap \P \setminus S_{L}\right) \bigcup \left(S_{K} \cap S_{L}\right).\] Hence, arguing as in the proof of Corollary \ref{MotivaElDelta},
\[\delta(T_{K,L})\leq 2\delta(S_{K}  \cap S_{L} )+1-\delta(S_{K})-\delta(S_{L})=2\delta(S_{\widetilde{K}}  \cap S_{\widetilde{L}} )+1-\delta(S_{\widetilde{K}})-\delta(S_{\widetilde{L}})\]
\[=2\delta(S_{\widetilde{KL}} )+1-\delta(S_{\widetilde{K}})-\delta(S_{\widetilde{L}})=\frac{2}{[\widetilde{K}\widetilde{L}:\Q]} +1 -\frac{1}{[\widetilde{K}:\Q]} -\frac{1}{[\widetilde{L}:\Q]}.\] The last claim follows from Corollary \ref{MotivaElDelta}.
\end{proof}

Motivated by the results we have obtained so far we define the quantity $\delta_{K,L}$.

\begin{definition}
Let $K,L$ be number fields with resp Galois closures $\widetilde{K}$ and $\widetilde{L}$.  Then, \[\delta_{K,L}:= \frac{2}{[\widetilde{K}\widetilde{L}:\Q]} +1 -\frac{1}{[\widetilde{K}:\Q]} -\frac{1}{[\widetilde{L}:\Q]}.\]
\end{definition}

\begin{proposition}
Let $K$ and $L$ be number fields. Then,

\[\frac{1}{[\widetilde{K}\widetilde{L}:\Q]}  \leq \delta_{K,L} \leq 1\]

where  equality on the right  occurs if and only if $\widetilde{K}=\widetilde{L}$ and equality on the left  occurs if and only if either $K$ or $L$ is equal to $\Q$.

\end{proposition}

\begin{proof}

The inequality of the left follows from Propositions \ref{PrimerDesigualdadA} and \ref{ALessThanD}. Clearly if either $K$ or $L$ is equal to $\Q$ then $\frac{1}{[\widetilde{K}\widetilde{L}:\Q]}  \leq \delta_{K,L}$. Conversely, if   
$\frac{1}{[\widetilde{K}\widetilde{L}:\Q]}  = \delta_{K,L}$ then \[0= \frac{1}{[\widetilde{K}\widetilde{L}:\Q]} +1 -\frac{1}{[\widetilde{K}:\Q]} -\frac{1}{[\widetilde{L}:\Q]}.\] Notice that  \[ \frac{1}{[\widetilde{K}\widetilde{L}:\Q]} +1 -\frac{1}{[\widetilde{K}:\Q]} -\frac{1}{[\widetilde{L}:\Q]}= \left(\frac{1}{[\widetilde{K}\widetilde{L}:\Q]} - \frac{1}{[\widetilde{K}:\Q][\widetilde{L}:\Q]}\right)+ \left(1- \frac{1}{[\widetilde{K}:\Q]}\right)\left(1- \frac{1}{[\widetilde{L}:\Q]}\right). \] Since both terms in the sum are non-negative it follows that they are both equal to $0$ which implies that either $K$ or $L$ is $\Q$. To show the remaining part notice that \[[\widetilde{K}:\Q][\widetilde{L}:\Q] \leq [\widetilde{KL}:\Q]^{2}\] and that equality occurs if and only if $\widetilde{K}=\widetilde{L}$. It follows from the AM-GM inequality that \[\frac{2}{[\widetilde{KL}:\Q]} \leq \frac{1}{[\widetilde{K}:\Q]} +\frac{1}{[\widetilde{L}:\Q]}\] with equality if and only if $\widetilde{K}=\widetilde{L}$. This is equivalent $\delta_{K,L} \leq 1$ and the claim about the occurrence of equality.
\end{proof}

\section{Naive approach}

Having defined $\delta_{K,L}$ we start a first approach to answer Question \ref{HayCota} based on its relation to $\delta(A_{K,L})$ and to similar quantities we define next.

\begin{remark}
Notice that $A_{K,L} \subseteq S_{K,L} \cap T_{K,L}$. In particular, $\delta(A_{K,L}) \leq \min \{  \delta(S_{K,L}),  \delta(T_{K,L}) \}$.
\end{remark}

\begin{corollary}\label{InequalityADelta}

Let $K$ and $L$ be degree $n$ number fields. Then, \[\delta(A_{K,L}) \leq \delta_{K,L}.\] 

\end{corollary}

\begin{proof}
Since $\delta(A_{K,L}) \leq \delta(T_{K,L})$ the result follows from Proposition \ref{ALessThanD}.
\end{proof}

Some of the classical results on Arithmetic equivalence can be phrased in terms of the above:

\begin{theorem}\label{AEGalRep}
Let $K$ and $L$ be number fields. The following are equivalent:

\begin{itemize}

\item[(i)] $\displaystyle \zeta_{K}(s)=\zeta_{L}(s)$

\item[(ii)] $\delta(A_{K,L})=1$

\item[(iii)] $\delta(T_{K,L})=1$

\item[(iv)]  $\delta(S_{K,L})=1$.

\end{itemize}

\end{theorem}

\begin{proof}

The implication (i) $\Rightarrow$ (ii) follows from equality of coefficients in convergent Dirichlet series. The implications (i) $\Rightarrow$ (iii), (iv) follow from the remark above. For (iii) $\Rightarrow$ (i) see \cite{Manti}. Lastly, for   (vi) $\Rightarrow$ (i) see \cite{Manti1} or \cite{PerStu}.
\end{proof}

\begin{lemma}\label{SalvadaPrimeDegreeGalois}
Let $K,L$ be Galois number fields of the same prime degree. Let $p$ be a rational prime not ramified in either field. Then, $a_{p}(K)=a_{p}(L)$ if and only if $A_{p}(K)=A_{p}(L)$. In particular, \[\delta(A_{K,L})=\delta_{K,L}.\]
\end{lemma}

\begin{proof}
Let $\ell$ be the common degree. Since both extensions are Galois we have that either $a_{p}(K)=a_{p}(L)=\ell$ or $a_{p}(K)=a_{p}(L)=0$. The former happens if and only if $p$ splits equivalently if and only if $A_{p}(K)=A_{p}(L)=(1,1,...,1)$. Since $\ell$ is prime the later occurs if and only there is only one prime laying over $p$ in either field, i.e., if and only if  $A_{p}(K)=A_{p}(L)=(\ell)$.  It follows that $\delta(A_{K,L})=\delta(T_{K,L})$, which concludes the proof thanks to the second part of Proposition \ref{ALessThanD}.
\end{proof}

These results about $ \delta_{K,L}$ give us a first clue on what to expect of an answer to Question \ref{HayCota}. Next we show some candidates for bounds in degrees $2$ and $3$.

\begin{theorem}\label{PrimeroQuadratics}

Let $\ell$ be a prime and let $K,L$ be Galois number fields of degree $\ell$. Suppose that $\delta(A_{K,L}) > 1-\frac{2}{\ell^2}$. Then, $K$ and $L$ are arithmetically equivalent. Furthermore, if $\ell=2$ and $\delta(A_{K,L}) = 1-\frac{1}{2^2}=\frac{1}{2}$ then $K$ and $L$ are not arithmetically equivalent.

\end{theorem}

\begin{proof}

Since any Galois number field is solitary, i.e., any other field with the same Dedekind zeta function is isomorphic to it, we must prove that the hypothesis implies that $K$ and $L$ are isomorphic. By Corollary \ref{InequalityADelta} we have that $\delta_{K,L}>1-\frac{2}{\ell^2}$. Thus, \[ \frac{2}{[KL:\Q]} +1 -\frac{1}{\ell} -\frac{1}{\ell} >1-\frac{2}{\ell^2},\] equivalently $\displaystyle \frac{\ell^2}{\ell-1}>[KL:\Q]$. This last inequality cannot be unless $K$ and $L$ are the same field, otherwise they would be linearly disjoint. To show the last part, notice that if $K$ and $L$ are not isomorphic then  $\delta_{K,L}=\frac{2}{4}+1 -\frac{1}{2} -\frac{1}{2} =\frac{1}{2}$. Since both fields are Galois of prime degree, $\delta(A_{K,L})=\delta_{K,L}$ thanks to Lemma \ref{SalvadaPrimeDegreeGalois}.
\end{proof}

\begin{theorem}\label{PrimeroCubics}

Let $K,L$ be cubic fields. Suppose that $\delta(A_{K,L}) > \frac{13}{18}$. Then, $K$ and $L$ are arithmetically equivalent.

\end{theorem}

\begin{proof}

If we show that $\delta_{K,L} \leq \frac{13}{18}$ whenever $K$ and $L$ are not isomorphic, the result follows from Corollary \ref{InequalityADelta} and from the fact that cubic fields are solitary. We divide the proof of $\delta_{K,L} \leq \frac{13}{18}$ in three cases:

\begin{itemize}

\item Both $K$ and $L$ are Galois: in such case $\delta_{K,L} =\frac{2}{9}+1-\frac{2}{3}=\frac{5}{9}$.

\item One of them is Galois and the other is not:  in such case $\delta_{K,L} =\frac{1}{9}+1-\frac{1}{6} -\frac{1}{3}=\frac{11}{18}$.

\item Neither of them is Galois:  in such case $\delta_{K,L} =\frac{1}{18}+1-\frac{2}{6}=\frac{13}{18}$.

\end{itemize}

\end{proof}

\subsubsection{A heuristic argument for a bound depending only on the degree.}

As we saw in Theorems \ref{PrimeroQuadratics}, \ref{PrimeroCubics}, $\delta_{K,L}$ played an important role in answering Question \ref{HayCota} in degrees $2$ and $3$.  For now on let us fix the degree $n$ of the number fields in question. In both theorems above we saw that there is a constant $b_{n}$, for $n=2,3$, such that if $\delta(A_{K,L}) > b_{n}$ then $K$ and $L$ are arithmetically equivalent. The argument in both cases was the same; $\delta(A_{K,L}) > b_{n}$ implies that $\delta_{K,L}  > b_{n}$ and, thanks to that in this specific situation the fields are solitary, arithmetic equivalence between them is equivalent to prove isomorphism. This suggests a strategy to look for a constant $b_{n}$ for all $n$. Since $\delta_{K,L} = \delta(T_{K,L})$ for $K$ and $L$ Galois, and Galois fields are solitary, we could try and see if there is a constant $b_{n}$ such that $\delta_{K,L} > b_{n}$ implies that $K$ and $L$ are isomorphic(in fact the same since they are Galois). Based on the two theorems above, and similar calculations done in Galois quartic fields, we observe that \[b_{n}=1-\frac{2}{n^2}\] is a good candidate. The following theorem confirms our choice.

\begin{theorem}

Let $n$ be positive integer and let $K$ and $L$ be Galois number fields of degree $n$.  Then,  \[\delta_{K,L} > 1-\frac{2}{n^2} \Longrightarrow K=L.\]

\end{theorem}

\begin{proof}

Since $K,L$ are degree n Galois we have by definition that $\displaystyle \delta_{K,L} =\frac{2}{[KL:\Q]}-\frac{2}{n}+1$. Therefore, after multiplying by $\frac{n^2}{2}$ and reordering,  $\delta_{K,L} > 1-\frac{2}{n^2}$ is equivalent to \[1 > \frac{n}{[KL:K]}([KL:K] -1).\] Since both extensions are Galois of degree $n$,  $[KL:K]$ is a divisor of $n$ hence the two factors on the right are non-negative integers which implies that $[KL:K]=1$, and thus $K=L$.
\end{proof}

\begin{remark}
Notice that $b_{3}=\frac{7}{9}$ is bigger that $\frac{13}{18}$, the constant found in Theorem \ref{PrimeroCubics},  hence it looks that there might be some room for improvement for the expression of $b_{n}$.  This seems not to be the case since for $n=2$ the constant $b_{2}$ is sharp.
\end{remark}

\section{Geometric approach}

In the previous section we proved that for degrees up to $n=3$  \[  \delta(A_{K,L}) >1-\frac{2}{n^2} \iff \zeta_{K}(s)=\zeta_{L}(s).\] This in fact can be improved to $n=4$ but requires a lot of ad hoc calculations depending on the behavior  of the Galois closures of quartic fields. Here we take a different approach to obtain a similar bound as the above, not as strong, but with the great advantage that the bound works for all $n$.\\

 The geometric approach we use here is to look at $\zeta_{K}(s)$ as the Weil-Zeta function of the Scheme $X_{K}:= {\rm Spec}(O_{K})$.  With this natural interpretation one can see, see \cite[Theorem 2.5]{Manti}, that the density of the set of primes $p$ for which $\#X_{K}(\F_{p})=\#X_{L}(\F_{p})$ is the same as $\delta(T_{K,L})$. Since \[\delta(A_{K,L}) \leq \delta(T_{K,L})\]
studying $\#X_{K}(\F_{p})=\#X_{L}(\F_{p})$ will gives us a way to obtain upper bounds for $\delta(A_{K,L})$. The following theorem is precisely the kind of result that we are looking for.

\begin{theorem}\label{ElDeSerre}\cite[Theorem 6.17]{Serre}

Let $X_{1}, X_{2}$ be separated schemes of finite type over $\Z$. Let S be a finitely set of primes such that $X_{1}$ and $X_{2}$ have good reduction away from $S$. Suppose that  $\#X_{1}(\F_{p})\neq \#X_{2}(\F_{p})$ for at least one prime $p \notin S$. Then, \[ \delta\left( \{ p: \#X_{1}(\F_{p})= \#X_{2}(\F_{p}) \} \right) \leq  1-\frac{1}{B^2}\] where $B:=B(X_{1})+B(X_{2})$ and \[B(X_{i})=\sum_{j} \dim_{\Q} {\rm H}^{j}_{c} (X_{i}(\C),\Q) \ \mbox{for} \ i=1,2. \]

\end{theorem}

Theorem \ref{ElDeSerre} provides the right setting for our purposes.

\begin{theorem}\label{ElPrincipal}
Let $n$ be a positive integers and let $K, L$ be degree $n$ number fields. Then, $K$ and $L$ are arithmetically equivalent if and only if the set of primes such that their arithmetic type is not the same in $K$ and $L$ has Dirichlet Density less that $\displaystyle \frac{1}{4n^2} $, i.e.,   \[  \delta(A_{K,L}) >1-\frac{1}{4n^2} \iff  \zeta_{K}(s)=\zeta_{L}(s).\]

\end{theorem}

\begin{proof}

Let $X_{1}:= {\rm Spec}(O_{K})$ and $X_{1}:= {\rm Spec}(O_{L})$. Since the ring of integers of a number fields is a finitely generated $\Z$-module both $X_{1}$ and $X_{2}$ are of finite type over $\Z$, and being affine they are separated. Notice that $X_{1}(\C)$ is a finite set of $n$ elements, corresponding to the complex embeddings of  $K$, same for $L$, hence $\dim_{\Q} {\rm H}^{0}_{c} (X_{i}(\C),\Q) =n$ and  $\dim_{\Q} {\rm H}^{j}_{c} (X_{i}(\C),\Q) =0$ for $j >0$. Therefore $B=B(X_{1})+B(X_{2})=2n$. We conclude from Theorem \ref{ElDeSerre} that if there is a prime $p$ not ramifying in either $K$ or $L$ such that $a_{p}(K)=\#X_{1}(\F_{p}) \neq \#X_{2}(\F_{p})=a_{p}(L)$ then   $\delta(T_{K,L}) \leq 1 -\frac{1}{4n^2}.$ Suppose that $\delta(A_{K,L}) >1-\frac{1}{4n^2}.$ Since $\delta(T_{K,L}) \ge \delta(A_{K,L})$ it follows from our previous argument that $a_{p}(K)=a_{p}(L)$ for all but finitely many primes, hence  $\zeta_{K}(s)=\zeta_{L}(s)$. The other implication is clear since $\zeta_{K}(s)=\zeta_{L}(s)$ implies that $\delta(A_{K,L})=1$.

\end{proof}

Theorem \ref{ElPrincipal} confirms, with a weaker bound, the heuristics presented in the previous section. Based on this result and the heuristics we conjecture the following.

\begin{conjecture}

The bound of $\displaystyle \frac{1}{4n^2} $ in Theorem \ref{ElPrincipal} can be improved to $\displaystyle \frac{2}{n^2}$.

\end{conjecture}

\begin{remark}

Notice that we have proved the conjecture in the cases $n=2,3$; this was done in part of Theorems \ref{PrimeroQuadratics}, \ref{PrimeroCubics}. 

\end{remark}

\section{Efective bounds on arithmetic equivalence}

In giving and answer to  Question \ref{HayCota} we looked at the function $\zeta_{K}(s)$ as the Weil-zeta function of the scheme ${\rm  Spec}O_{K}$. Now, for Question \ref{CheboEfectivo} we see  $\zeta_{K}(s)$ as the Artin L-function of Galois representation $\rho_{K}: {\rm G}_{\Q}  \to {\rm GL}_{n}(\C).$ With this in hand  Question \ref{CheboEfectivo} can be answered by applying an effective version of Chebotarev's density theorem. 

\subsubsection{Dedekind zeta function as Artin L-function of a representation of  ${\rm G}_{\Q}$}

Let $K$ be a degree $n$ number field, and let us denote by $\widetilde{K}$ its Galois closure over $\Q$. We start by recalling the construction of an $n$-dimensional complex Galois representation $\rho_{K}$ of the absolute Galois group $G_{\Q}$ such that the Artin $L$-function associated to $\rho_{K}$ is $\zeta_{K}(s)$. Let ${\rm Emb}(K)$ be the set of complex embeddings of $K$. The absolute Galois group $G_{\Q}$ acts continuously on ${\rm Emb}(K)$ via composition. The continuity follows since the kernel of the action is the open group $G_{\widetilde{K}}$. Since   $n=\#{\rm Emb}(K)$ the above gives a continuous permutation representation   $G_{\Q}: \pi_{K} \to S_{n}$, which by composition with the permutation representation $\iota_{n}: S_{n} \to \rm{GL}_{n}(\C)$ produces an $n$-dimensional complex representation \[\rho_{K} : G_{\Q} \to \rm{GL}_{n}(\C).\]

\begin{definition}\label{TateModuleNumberField}
Let $K$ be  a number field. The continuous $\C[G_{\Q}]$-module $T_{K}$ is the $G_{\Q}$-module attached to the representation $\rho_{K}$. In other words, $\displaystyle T_{K}:=\bigoplus_{ \sigma \in {\rm Emb} } \C\sigma$ with the action of $G_{\Q}$ on each element of the basis given by composition.
\end{definition}

The relevance of this representation to our purposes is that the Artin formalism gives us the following:

\begin{proposition}\label{ZetaLArtin}
Let $K$ be a number field and let us denote by $L(\rho, s)$ the Artin L-function attached to a representation $\rho$. Then $\displaystyle L(\rho_{K}, s)= \zeta_{K}(s).$
\end{proposition}

\begin{proof}

By Galois correspondence $\rho_{K}$ factorizes through a map $\widetilde{\rho}_{K}:  {\rm Gal}(\widetilde{K}/\Q) \to \rm{GL}_{n}(\C)$

\[\xymatrix{
{G_{\Q}}\ar[r]^{{\rm Res}^{\overline{\Q}}_{\widetilde{K}} \ \ \ \ \ \  }\ar[dr]_{\rho_{K}}&{ {\rm Gal}(\widetilde{K}/\Q)}\ar[d]^{\widetilde{\rho}_{K}}\cr
&{\rm{GL}_{n}(\C)}\cr}
\]

 Again, by basic Galois theory, the action of ${\rm Gal}(\widetilde{K}/\Q)$ in ${\rm Emb}(K)$ is isomorphic to the permutation representation of  ${\rm Gal}(\widetilde{K}/\Q)$ in the set of cosets  ${\rm Gal}(\widetilde{K}/\Q)/ {\rm Gal}(\widetilde{K}/K)$. Hence, 
$\displaystyle \widetilde{\rho}_{K} \cong  {\rm Ind}_{{\rm Gal}(\widetilde{K}/K)}^{{\rm Gal}(\widetilde{K}/\Q)} 1_{{\rm Gal}(\widetilde{K}/K)}.$ Thanks to Artin's formalism 
\[ L(\rho_{K}, s)=L\left(\widetilde{\rho}_{K}, s\right) =L\left( {\rm Ind}_{{\rm Gal}(\widetilde{K}/K)}^{{\rm Gal}(\widetilde{K}/\Q)} 1_{{\rm Gal}(\widetilde{K}/K)} , s\right)=L(1_{{\rm Gal}(\widetilde{K}/K)} ,s)= \zeta_{K}(s).\] \end{proof}

Since the Dedekind zeta function is an Artin $L$-function then its prime terms correspond to traces of Frobenius elements:

\begin{corollary}\label{Frobenius}

Let $K$ be a number field and $\ell$ be a prime unramified\footnote{This is the same as being unramified in $K$ since the conductor of $\rho_{K}$ is the discriminant of $K$.} under $\rho_{K}$. Let ${\rm Frob}_{\ell}$ be the conjugacy class of the element Frobenius at $\ell$. Then,  \[{\rm Trace}(\rho_{K}({\rm Frob}_{\ell})) =a_{\ell}(K).\]
\end{corollary}

Proposition \ref{ZetaLArtin} gives not only a simple way to express the trace of Frobenius but it also gives a useful generalization of the above corollary to calculate its characteristic polynomial $\displaystyle \det(X-\rho_{K}({\rm Frob}_{\ell})).$

\begin{lemma}
Let $K$ be a number field and $\ell$ be a prime, unramified in $K$, and let $(f_{1},...,f_{g})$ be the arithmetic type of $\ell$ in $K$. Then, 
\[{\rm det}(X-\rho_{K}({\rm Frob}_{\ell}))=\prod_{i=1}^{g}(X^{f_{i}}-1).\]

\end{lemma}

For a proof see \cite[Lemma 2.4]{Manti}.\\

This type of analysis provides a way to obtain several of the classic results about arithmetic equivalence. For instance, using Corollary  \ref{Frobenius} the equivalence (i), (iii) of Theorem \ref{AEGalRep} \[\delta(T_{K,L})=1 \iff \zeta_{K}(s)=\zeta_{K}(s)\] is obtained. 

This point of view together with an effective version of Chebotarev's density theorem gives an answer to Question \ref{CheboEfectivo}. For details about an effective version of Chebotarev's density theorem, and a proof of the following, see \cite{AhnKwon}.

\begin{theorem}\label{ElCheboEfectivo}
Let $E$ be a Galois number field with Galois group $G=\gal(E/\Q)$ and absolute value discriminant $D$ and degree $N$.  For every conjugacy class $C$ of $G$ there is a rational prime $p$, unramified in $E$, such that  $p \leq D^{12577} $ and the conjugacy class of a Frobenius element at $p$ is equal to $C$. Under GRH the bound $D^{12577}$ can be replaced by \[(4\log(D)+2.5N+5)^2. \]
\end{theorem}

\begin{definition}\label{LaConstanteCheboEfectiveField}

Let $K$ be a number field with Galois closure $\widetilde{K}$. Suppose that $\widetilde{K}$ has degree $n_{\widetilde{K}}$ and absolute value of the discriminant equal to $d_{\widetilde{K}}$. Let \[X_{K}:=d_{\widetilde{K}}^{12577} \ \ \ X_{K}^{1}:=(4\log(d_{\widetilde{K}})+2.5n_{\widetilde{K}}+5)^2 \]

\end{definition}

\begin{theorem}\label{AECheboEfectiva}
Let $K$ be a number field. Then, for every number field $L$ such that $\widetilde{K}=\widetilde{L}$ we have that \[ \zeta_{K}(s)=\zeta_{L}(s)  \iff A_{p}(K)=A_{p}(L)\  \mbox{for all}  \ p \leq X_{K}.\]

\noindent Moreover, under GRH the bound $X_{K}$ can be replaced by $X_{K}^{1}.$

\end{theorem}

\begin{proof}

Suppose that $A_{p}(K)=A_{p}(L)\  \mbox{for all}  \ p \leq X_{K}$(resp. $X^{1}_{K}$). Since $a_{p}(K)$ is equal to the number of 1's in $A_{p}(K)$, resp. for $L$, then $a_{p}(K)=a_{p}(L)$ for all $p \leq X_{K}$ resp. $X^{1}_{K}$. Let $\sigma \in G_{\Q}$ and let $\widetilde{\sigma} \in \gal(\widetilde{K}/\Q)$ be the restriction of $\sigma$ to $\widetilde{K}$. By Theorem \ref{ElCheboEfectivo} there is a prime $p \leq X_{K}$ (resp. $X^{1}_{K}$ under GRH) such that $\widetilde{\sigma}$ is the conjugacy class of a Frobenius element in  $\gal(\widetilde{K}/\Q)$  at $p$, say ${\rm Frob}_{p}^{\widetilde{K}}$. Then, since $\displaystyle \widetilde{\rho}_{E} \circ {\rm Res}^{\overline{\Q}}_{\widetilde{E}}=\rho_{E}$, for $E=K,L$, and since $\widetilde{K} =\widetilde{L}$ we have that 
\[ {\rm Trace}(\rho_{K}(\sigma))={\rm Trace}(\widetilde{\rho}_{K}(\widetilde{\sigma}))= {\rm Trace}(\widetilde{\rho}_{K}({\rm Frob}_{p}^{\widetilde{K}}))=a_{p}(K)=\] \[a_{p}(L)={\rm Trace}(\widetilde{\rho}_{L}({\rm Frob}_{p}^{\widetilde{K}}))={\rm Trace}(\widetilde{\rho}_{L}(\widetilde{\sigma}))={\rm Trace}(\rho_{L}(\sigma)).\] Since $\sigma$ is an arbitrary element, and $\rho_{K}$ and $\rho_{L}$ are Artin representations, we have that $\rho_{K}$ and $\rho_{L}$ are equivalent so their $L$-functions are the same, i.e., $\zeta_{K}(s)=\zeta_{L}(s)$ thanks to Proposition \ref{ZetaLArtin}. The reverse implication is clear by Theorem \ref{Perlis}.

\end{proof}

%\guillermo{TWO SEPTIC FIELDS CALCULATION} 

\begin{remark}

Thanks to a result of Zaman, see \cite{Asif}, the bound $d_{\widetilde{K}}^{12577}$  can be replaced by a bound of the form $Cd_{\widetilde{K}}^{40}$ where $C$ is an effective absolute constant.

\end{remark}

{\footnotesize Guillermo Mantilla-Soler, Department of Mathematics, Fundaci\'on Universitaria Konrad Lorenz,
Bogot\'a, Colombia. Department of Mathematics and Systems Analysis, Aalto University, Espoo, Finland. ({\tt gmantelia@gmail.com})}

\end{document}